\documentclass[english, 10pt]{amsart}

\usepackage{babel}
\usepackage[utf8x]{inputenc}
\usepackage{mathrsfs}
\usepackage{amsmath}
\usepackage{amssymb}
\usepackage{amsthm}
\usepackage[all]{xy}
\usepackage{leftidx}
\usepackage{amsmath,amssymb,amsfonts}
\usepackage{amsmath,amscd}
\usepackage{textcomp}

\newtheorem{defn}{Definition}
\newtheorem{thm}{Theorem}[section]

\newtheorem{lemme}[thm]{Lemma}
\newtheorem{remark}[defn]{Remark}

\newtheorem{note}[defn]{Notations}

\newtheorem{cor}[thm]{Corollary}

\DeclareMathOperator{\Div}{Div}

\DeclareMathOperator{\Id}{Id}

\DeclareMathOperator{\rank}{rank}

\DeclareMathOperator{\supp}{supp}

\DeclareMathOperator{\psh}{psh}

\DeclareMathOperator{\End}{End}
\DeclareMathOperator{\pr}{pr}

\newcommand\sI{{\mathcal I}}

\newcommand\ls{{\epsilon}}

\title[Ohsawa-Takegoshi extension theorem]{Ohsawa-Takegoshi extension theorem for compact K\"ahler manifolds and applications}

\date{\today}

\author{Junyan Cao}


\address{Junyan Cao, 
Université Paris 6,  Case 247,  
Institut de Mathématiques de Jussieu, Analyse complexe et géométrie, 
4, Place Jussieu,
France}
\email{junyan.cao@imj-prg.fr}

\begin{document}

\begin{abstract}
Our main goal in this article is to prove an extension theorem
for sections of the canonical bundle of a weakly pseudoconvex K\"ahler manifold
with values in a line bundle endowed with a possibly singular metric.
We also give some applications of our result.
\end{abstract}

\maketitle

\section{Introduction}
 
The $L^2$ extension theorem by Ohsawa-Takegoshi is 
a tool of fundamental importance in algebraic and analytic geometry. After the crucial contribution
of \cite{OT87, Ohs88}, this result 
has been generalized by many authors in various contexts, including 
\cite{Man93}, \cite{Dem00}, \cite{Siu04}, \cite{Ber05}, \cite{Che11}, \cite{Yi12}, \cite{ZGZ12}, \cite{Blo13}, \cite{GZ15a},\cite{Dem15}, \cite{BL16}.
\medskip

\noindent In this article we treat yet another version of the extension theorem in the context of
 K\"ahler manifolds. 
We first state a consequence of our main result; we remark that a version of it was conjectured by Y.-T. Siu in the framework of his work on the invariance of plurigenera. 

\begin{thm}\label{optimalestimateintro}
Let $(X, \omega)$ be a K\"ahler manifold and $\pr: X\rightarrow \Delta$ be a proper holomorphic map 
to 
the ball $\Delta\subset \mathbb{C}^1$ centered at $0$ of radius $R$. 
Let $(L, h)$ be a holomorphic line bundle over $X$ equipped with 
a hermitian metric (maybe singular) $h = h_0 e^{-\varphi_L}$ such that $i\Theta_h (L)\geq 0$ in the sense of currents,
where $h_0$ is a smooth hermitian metric and $\varphi_L$ is a quasi-psh function over $X$.
We suppose that $X_0 := \pr^{-1}(0)$ is smooth of codimension $1$, and that the restriction of 
$h$ to $X_0$ is not identically $\infty$. 

Let $f\in H^0 (X_0 , K_{X_0} \otimes L)$ be a holomorphic section in the multiplier ideal defined by the restriction of $h$ to $X_0$.
Then there exists a section $F \in H^0 (X, K_{X} \otimes L)$ whose restriction to $X_0$ is equal to
$f$, and  
such that the following optimal estimate holds
\begin{equation}
\frac{1}{\pi R^2}\int_{X} |F|_{\omega, h} ^2 d V_{X, \omega} \leq  \int_{X_0} |f|_{\omega, h} ^2 d V_{X_0, \omega}.
\end{equation}
\end{thm}

\noindent We note that the volume form $|F|_{\omega, h} ^2  d V_{X, \omega}$ is independent of choice of the metric $\omega$, and $d V_{X_0, \omega}$
is the volume form on $X_0$ induced by the metric $\omega |_{X_0}$.

\medskip

If the manifold $X$ is isomorphic to the product $X_0\times \Delta$
and if the line bundle $L$ is trivial, then it is clear how to construct $F$. If not, the existence of an extension verifying the estimate above is quite 
subtle, and it has many important applications. The result above is proved by combining the arguments in \cite{Blo13}, \cite{GZ15a} and \cite{Yi12}. 
Comparing to \cite{Blo13, GZ15a}, the new input here is that we allow the metric $h$ of $L$ to be singular,
while the ambient manifold is only assumed to be K\"ahler. This general context leads to rather
severe difficulties, mainly due to the loss of positivity in the process of regularizing the metric $h$
which adds to the intricate relationship between the several parameters involved in the proof.  
We use here the arguments in \cite{Yi12} to overcome the difficulties.

Before stating the main result of this paper in its most general form and explaining the main ideas in the proof, 
we note the following consequence of Theorem \ref{optimalestimateintro} by an idea of H. Tsuji.
It is a generalization of \cite[Thm 0.1]{BP10} to arbitrary compact K\"ahler families, which follows from 
our main theorem and the arguments in \cite[Cor 3.7]{GZ15a}.

\begin{thm}\label{introprop}
Let $p: X\rightarrow Y$ be a fibration between two compact K\"ahler manifolds.
Let $L\rightarrow X$ be a line bundle endowed with a metric (maybe singular) $h = h_0 e^{-\varphi_L}$ such that $i\Theta_h (L)\geq 0$ in the sense of currents,
where $h_0$ is a smooth hermitian metric and $\varphi_L$ is a quasi-psh function over $X$.
Suppose that there exists a generic point $z\in Y$ and a section $u\in H^0 (X_z, m K_{X/Y}+L)$ such that 
$$\int_{X_z} |u|_{\omega, h} ^{\frac{2}{m}} d V_{X_z ,\omega} < +\infty .$$
Then the line bundle $m K_{X/Y} +L$ admits a metric with positive curvature current. 
Moreover, this metric equals to the fiberwise $m$-Bergman kernel metric on the generic fibers of $p$.
\end{thm}
\smallskip

\noindent We note that the original proof of the theorem above in the projective case
does not go through in the K\"ahler case. This is due to the fact that in \cite[Thm 0.1]{BP10}
the authors are using in an essential manner the existence of Zariski dense open subsets of $X$.
\medskip

We will state next our general version of Theorem \ref{optimalestimateintro}; prior to this, we introduce some auxiliary weights, following \cite{Blo13}, \cite{GZ15a}.

\begin{note} 
Given $\delta >0$ and $A\in \mathbb{R}$, 
let $c_A (t)$ be a positive smooth function on $(-A, +\infty)$ such that $\int_{-A} ^{+\infty} c_A (t) e^{-t} d t < +\infty$.
Set 
$$u(t) = -\ln (\frac{c_A (-A) e^{A}}{\delta} +\int_{-A} ^t c_A (t_1) e^{-t_1} dt_1) ,$$
and
$$s (t) = \frac{\int_{-A} ^t e^{-u(t_1)} d t_1 + \frac{c_A (-A) e^A}{\delta^2}}{ e^{-u(t)}}.$$
Then $u (t)$ and $ s(t)$ satisfy the ODE equations:
\begin{equation}\label{ODEequation1}
(s (t)+ \frac{(s' (t))^2}{u''(t) s (t) -s'' (t)}) e^{u (t) - t}=\frac{1}{c_A (t)}
\end{equation}
and
\begin{equation}\label{ODEequation2}
s'(t) - s(t) u' (t)=1.
\end{equation}
We suppose moreover that 
\begin{equation}\label{addcondition}
e^{- u(t)} \geq c_A (t) s(t) \cdot e^{-t}\qquad \text{for every }t\in (-A, +\infty) .
\end{equation}
\end{note}

\begin{remark}\label{defnweight}
If $c_A (t)\cdot e^{-t}$ is decreasing, then \eqref{addcondition} is automatically satisfied.
Moreover, by the construction of $u (t), s(t)$, we know that
\begin{equation}\label{estimateu}
\lim\limits_{t\rightarrow +\infty} u (t) =  -\ln (\frac{c_A (-A) e^{A}}{\delta} +\int_{-A} ^{+\infty} c_A (t_1) e^{-t_1} dt_1) < +\infty
\end{equation}
and 
\begin{equation}\label{estimates}
|s (t)|\leq  C_1 |t| + C_2 
\end{equation}
for two constants $C_1, C_2$ independent of $t$.
\end{remark}

\medskip 

\noindent In this set-up, by combining the arguments in \cite{GZ15a} and \cite{Yi12}, we can prove the main result of the present paper:
 
\begin{thm}\label{mainthmoptimal}
Let $(X,\omega )$ be a weakly pseudoconvex $n$-dimensional K\"ahler manifold and 
$E$ be a vector bundle of rank $r$ endowed with a smooth metric $h_E$.
Let $Z\subset X$ be the zero locus of $v\in H^0 (X, E)$. 
We assume that $Z$ is smooth of codimension $r$ and $|v|_{h_E} ^{2r} \leq e^A$ for some $A\in\mathbb{R}$.
Set $\Psi (z) := \ln |v|_{h_E} ^{2r}$.

Let $L$ be a line bundle on $X$ equipped with a singular metric $h:= h_0 \cdot e^{-\varphi}$
such that $i \Theta_{h} (L) \geq \gamma$ for some continuous $(1,1)$-form $\gamma$, where $h_0$ is a smooth metric on $L$.
We assume that there exists
a sequence of analytic approximations $\{ \varphi_k \}_{k=1}^{\infty}$ of $\varphi$ such that
\footnote{If $X$ is compact, such approximation always exists, cf. \cite[Chapter 13]{Dem12}}
\begin{equation}\label{approximation}
i \Theta_{h_0 \cdot e^{-\varphi_k}} (L)\geq \gamma -\frac{\omega}{k} .
\end{equation}
We suppose that there exists a continuous function $a(t)$on $(-A, +\infty]$, such that $0 < a(t) \leq s(t)$ and
\begin{equation}\label{addedlast}
a(-\Psi) (\gamma +  i d'd'' \Psi ) +  i d' d'' \Psi \geq 0.
\end{equation}

Then for every $f\in H^0 (Z, K_X\otimes L \otimes \sI (\varphi|_Z))$\footnote{$\sI (\varphi|_Z)$ is the multiplier ideal sheaf on $Z$ associated 
to the weight $\varphi|_Z$}, there exists a $F\in H^0 (X, K_X\otimes L)$ such that 
$F\mid_Z =f$ and
\begin{equation}\label{normcontrol}
\int_X c_A (-\Psi ) |F|_{\omega, h} ^2 dV_{X,\omega} 
\leq e^{-\lim\limits_{t\rightarrow +\infty} u(t)} \int_Z \frac{|f|_{\omega, h} ^2 } {|\Lambda^r (d v)|^2} dV_{Z,\omega} ,
\end{equation}
where the weight $|\Lambda^r (d s)|^2$ is defined as the unique function such that
$$\int_Z \frac{G } {|\Lambda^r (d v)|^2} dV_{Z,\omega} 
=\lim_{m\rightarrow +\infty} \int_{ -m-1\leq \ln |v|^{2r} _{h_E} \leq -m} \frac{G}{|v|^{2r} _{h_E}} dV_{X,\omega} \qquad\text{for every }G\in C^{\infty}(X).$$
\end{thm}

\medskip

\begin{remark}\label{specialcaseremark}
As already pointed out in \cite{GZ15a}, by taking $E=\pr^* \mathcal{O}_{\Delta}$, $v=\pr^* z$, $A=\ln R^2$,
$c_A (t)\equiv 1$ and letting $\delta\rightarrow +\infty$,
Theorem \ref{mainthmoptimal} implies Theorem \ref{optimalestimateintro}.
\end{remark}

\bigskip

We comment next a few results at the foundation of Theorem \ref{mainthmoptimal}. 
The original Ohsawa-Takegoshi extension theorem \cite{OT87} deals with the local case,
i.e. $X$ is a pseudoconvex domain in $\mathbb{C}^n$.
The potential applications of this type of results 
in global complex geometry become apparent shortly after the article \cite{OT87}
appeared, and to this end it was necessary 
to rephrase it in the context of manifolds.
As far as we are aware, the first global version is due to L. Manivel \cite{Man93}.
We quote here a simplified version of his result.

\begin{thm}\label{smoothcase}\cite[Thm 2]{Man93}
Let $X$ be a $n$-dimensional Stein manifold,
and $E $ be a holomorphic vector bundle over $X$ of rank $r$ with a smooth metric $h_E$.
Let $Y\subset X$ be the zero locus of $s\in H^0 (X, E)$. 
We assume that $Y$ is smooth and of codimension $r$. 
Let $\Omega$ be a $(1,1)$-closed semi-positive form on $X$ such that 
$$\Omega\otimes \Id_E \geq i\Theta_{h_E} (E)$$
in the sense of Griffiths, i.e., $\Omega\otimes \Id_E - i\Theta_{h_E} (E)$ is semipositive on the vectors 
$\xi \otimes s \in T_X \otimes \End (E)$ for every $\xi \in T_X$ and $s\in \End (E)$.

Let $(L, h_0)$ be a line bundle on $X$ equipped with a smooth metric $h_0$,
such that there exists a constant $\alpha > 0$ satisfying  
$$i \Theta_{h_0} (L) \geq \alpha \Omega - r i d'd'' \ln |s|^2 _{h_E} .$$
Then for every $f\in H^0 (Y, K_Y\otimes L\otimes (\det E)^{-1})$, 
there exists a section $F\in H^0 (X, K_X\otimes L)$ such that 
$F\mid_Y =f \wedge (\wedge^r ds)$ and
\begin{equation}\label{l2}
 \int_X \frac{|F|_{\omega, h_0} ^2}{|s|_{h_E} ^{2r-2} ( 1+|s|_{h_E} ^2)^{\beta}} dV_{X,\omega}\leq 
 C \int_Y |f|_{\omega, h_0} ^2  dV_{Y,\omega},
\end{equation}
where $C$ is a numerical constant depending only on $r$, $\alpha$ and $\beta$.
\end{thm} 

\begin{remark}\label{adddDem}
Theorem \ref{smoothcase} can 
be easily generalized to the case when $X$ is a weakly pseudoconvex K\"ahler manifold and 
the weight function $|s|_{h_E} ^{2r-2} ( 1+|s|_{h_E} ^2)^{\beta}$ can be ameliorated by 
$|s|_{h_E} ^{2r} ( \ln |s|_{h_E}) ^2$, cf. \cite[Thm 12.6]{Dem12}.
\end{remark}

One of the important limitations of Theorem \ref{smoothcase}
is that the metric $h_0$ is assumed to be smooth. 
Indeed this is unfortunate, given that in the usual set-up of algebraic geometry one has to deal
with extension problems for canonical forms with values in pseudo-effective line bundles.
A famous example is the invariance of plurigenera for projective manifolds (\cite{Siu04}): 
one needs an extension theorem under the hypothesis that the metric $h_0$ has arbitrary singularities.
We remark that the proof of the extension theorem used in the article mentioned above is confined to the case of projective manifolds. Thus,
in order to generalize \cite{Siu04} to compact K\"ahler manifolds, the first step 
would be to allow the metric $h_0$ in Theorem \ref{smoothcase}
to have arbitrary singularities.

Among the very few results in this direction we mention the important work of L. Yi.
In order to keep the discussion simple, we restrict ourselves to the setup in
Theorem \ref{optimalestimateintro}.
Let $\mathcal{I}_+ (h) :=\lim\limits_{\delta\rightarrow 0^+} \mathcal{I} (h ^{1+\delta})$.
L. Yi \cite{Yi12} established Theorem \ref{optimalestimateintro} \footnote{\cite{Yi12} proved it in a more general setting.} 
for sections $f$ which belong to the augmented multiplier ideal sheaf $\mathcal{I}_+ (h)$.
Guan and Zhou \cite{GZ15b} (cf. also Hiep \cite{Hie14}) 
showed that $\mathcal{I}_+ (h)= \mathcal{I} (h)$. Thus, the conjunction of these two
results as well as the optimal extension \cite{GZ15a} establish Theorem \ref{optimalestimateintro}.
The proof of our main theorem is mainly based on the arguments in \cite{GZ15a} and \cite{Yi12}.

\begin{remark}
In the situation of Theorem \ref{mainthmoptimal}, if we take the weight function $c_A (t) \equiv 1$, 
then we have Theorem \ref{optimalestimateintro}. 
There is another weight function which might be useful.
If we take $c_A (t)= \frac{e^t}{(t+A +c)^2}$ for some constant $c >0$, thanks to Remark \ref{defnweight}, \eqref{addcondition} is satisfied.
Using this weight function, \cite[Thm 3.16]{GZ15a} proved an optimal estimate version of Theorem \ref{smoothcase} and its remark \ref{adddDem}.
Thanks to Theorem \ref{mainthmoptimal}, we know that \cite[Thm 3.16]{GZ15a} is also true for weakly pseudoconvex K\"ahler manifolds 
under the approximation assumption \eqref{approximation}.
\end{remark}

\noindent {\bf Acknowledgements: } I would like to thank H. Tsuji who brought me attention to this problem during the Hayama conference 2013.
I would also like to thank M. P\u{a}un for pointing out several interesting applications, 
and a serious mistake in the first version of the article.
I would also like to thank J.-P. Demailly and X. Zhou for helpful discussions.
Last but not least, I would like to thank the anonymous referee for excellent suggestions about this work.

\section{Proof of Theorem \ref{mainthmoptimal} }\label{II}

\begin{proof}[Proof of Theorem \ref{mainthmoptimal}]
The constants $C_1, C_2,\cdots $ below are all independent of $k$.
The proof follows closely \cite{GZ15a} and \cite{Yi12}.
To begin with, we introduce several notations.
In the setting of Theorem \ref{mainthmoptimal}, 
for every $m\in \mathbb{R}$ fixed, we can define a $C^1$-function $b_{m}$ on $\mathbb{R}$ such that 
$$b_{m} (t) =t \text{ for } t\geq -m \qquad \text{and}\qquad b_{m} '' (t) = \textbf{1}_{\{-m -1\leq t\leq -m \}} .$$
Then 
\begin{equation}\label{estimatebm}
b_m (t) \geq t\qquad\text{and} \qquad b_m (t) \geq -m-1 \qquad\text{for every }t\in \mathbb{R}^1. 
\end{equation}
Let $s, u$ be the two functions defined in the introduction.
Set $\chi_m (z) := -b_m \circ \Psi$, $\eta_{m} (z) := s \circ \chi_m $ and $\phi_m (z) := u \circ \chi_m$.
Thanks to \eqref{estimateu} and \eqref{estimates}, we have
\begin{equation}\label{estimateweight1}
|\phi_m (z)|\leq C_1
\end{equation}
and 
\begin{equation}\label{estimateweight2}
 |\eta_m (z)|\leq C_2 |\chi_m (z)| + C_3 \leq C_2 \cdot \min  \{2r |\ln |v|_{h_E}|, m +1 \} +C_3.
\end{equation}
Set $\lambda_m (z):=\frac{(s') ^2}{u'' s-s''} \circ \chi_m$,
$h_k := h_0 \cdot e^{-\varphi_k}$ and $\widetilde{h}_{m, k} := h_k \cdot e^{-\Psi -\phi_m}$.
By \eqref{ODEequation1}, we have
\begin{equation}\label{importantequality11}
c_A (\chi_m) \cdot e^{-\chi_m +\phi_m }= (\eta_m + \lambda_m )^{-1}.
\end{equation}

The proof of the theorem is divided by three steps.

\medskip

\textbf{Step 1: Construction of smooth extension}

We construct in this step a smooth section $\widetilde{f} \in C^{\infty} (X, K_X \otimes L)$ extending $f$
such that  $(D'' \widetilde{f} )(z) =0$ for every $z\in Z$ and 
\begin{equation}\label{l2controlnew}
\int_X \frac{|D^{''}\widetilde{f} |^2 _{\omega, h_0}}{|v|_{h_E} ^{2r} (\ln |v|_{h_E} )^2 } \cdot e^{- (1+ \sigma)\varphi} dV_{X, \omega} 
\leq C_1 \cdot\int_Z \frac{|f|^2 _{\omega, h}}{|\Lambda^r (d v)|^2} dV_{Z, \omega} 
\end{equation}
for some constant $\sigma > 0$.

\medskip

In fact, let $ ( U_i )$ be a small Stein cover of $X$ and let $(\chi_i )$ be a partition of unity subordinate to $(U_i)$. 
Thanks to \cite{GZ15b}, there exists a $\sigma > 0$, such that 
$$\int_{U_i \cap Z} |f|^2 _{\omega, h_0} e^{-(1+\sigma)\varphi} d V_{Z,\omega}\leq 2 \int_{U_i \cap Z} |f|^2 _{\omega, h} d V_{Z,\omega} .$$
Applying the local Ohsawa-Takegoshi extension theorem (cf. for example \cite[Thm 12.6]{Dem12}) 
to the weight $e^{-(1+\sigma)\varphi}$ on $U_i$, 
we obtain a holomorphic section $f_i$ on $U_i$
such that 
\begin{equation}\label{l2cont}
\int_{U_i} \frac{|f_i |^2 _{\omega, h_0}}{|v|^{2r} _{h_E} (\ln |v|_{h_E} )^2} \cdot e^{-(1+\sigma)\varphi} dV_{X,\omega}
\leq  C_2 \cdot \int_{U_i \cap Z} \frac{|f|^2 _{\omega, h} }{|\Lambda^r (d v)|^2} dV_{Z,\omega}. 
\end{equation}
Set $\widetilde{f} := \sum\limits_{i} \chi_i \cdot f_i $. Then
$$(D^{''} \widetilde{f})  |_{U_j} = D^{''} (\sum_i \chi_i \cdot (f_i-f_j)) = \sum_i (\overline{\partial} \chi_i)\cdot (f_i-f_j) \text{ on } U_j .$$
Combining this with \eqref{l2cont}, \eqref{l2controlnew} is proved. We have also $(D^{''} \widetilde{f}) (z) =0$ for every $z\in Z$.

\bigskip

\textbf{Step 2: $L^2$ estimate}

Set $g_m := D^{''}((1- b_m ' \circ \Psi) \cdot  \widetilde{f})$. We claim that

\smallskip

{\em Claim:} There exists a sequence $\{a_m \}_{m=1}^{+\infty} \subset \mathbb{N}$ tending to $+\infty$, $\gamma_m$ and $\beta_m$ such that
\begin{equation}\label{partialequation}
D^{''} \gamma_m + (\frac{m}{a_m})^{\frac{1}{2}} \beta_m = g_m ,\qquad \lim_{m\rightarrow +\infty} \frac{m}{a_m} =0 ,
\end{equation}
and
\begin{equation}\label{keyestimateoptial}
\varlimsup_{m\rightarrow +\infty} (\int_{X} \frac{|\gamma_m|^2 _{\omega, \widetilde{h}_{m, a_m}}}{ \eta_{m} + \lambda_{m}} dV_{X,\omega} +
C\int_{X} | \beta_m |^2 _{\omega, \widetilde{h}_{m, a_m}} dV_{X,\omega})  
\end{equation}
$$\leq 
e^{-\lim\limits_{t\rightarrow +\infty} u(t)} \cdot \int_Z \frac{|f|_{\omega, h}^2} {|\Lambda^r (d v)|^2} dV_{Z,\omega}$$
for some uniform constant $C >0$.
The proof of the claim combines the estimates in \cite{GZ15a} and \cite{Yi12}.
We postpone the proof of the claim in Lemma \ref{L2estimateoptimal} and first finish the proof of the theorem.

\medskip

We use \eqref{keyestimateoptial} to estimate $\int_{X} c_A (-b_m\circ\Psi)\cdot  |\gamma_m |^2_{\omega, h_{a_m}} d V_{X,\omega}$.
By \eqref{estimatebm} and \eqref{importantequality11}, we have
$$c_A (- b_m\circ\Psi)\cdot e^{ \Psi +\phi_m }= c_A (\chi_m)\cdot e^{ \Psi +\phi_m } \leq (\eta_m + \lambda_m)^{-1}.$$
Therefore
\begin{equation}\label{secondestimateoptial}
\int_{X} c_A (- b_m\circ\Psi) \cdot  |\gamma_m |^2_{\omega, h_{a_m}} d V_{X,\omega} 
\leq \int_{X} \frac{|\gamma_m |^2 _{\omega, \widetilde{h}_{m,a_m}}}{(\eta_{m} + \lambda_m)} d V_{X,\omega} .
\end{equation}
Combining this with \eqref{keyestimateoptial}, we get
\begin{equation}\label{addestimate}
\varlimsup_{m\rightarrow +\infty}\int_X  c_A (- b_m \circ\Psi)  |\gamma_m |^2_{\omega, h_{a_m}} d V_{X,\omega} 
\leq e^{-\lim\limits_{t\rightarrow +\infty} u(t)} \cdot \int_Z \frac{|f|_{\omega, h_0}^2 e^{-\varphi_k}} {|\Lambda^r (d v)|^2} dV_{Z,\omega}.
\end{equation}
Thanks to \eqref{addestimate}, by passing to a subsequence, we can assume that the sequence 
$$\{\gamma_m -(1-b_m '\circ \Psi)\widetilde{f} \}_{k=1}^{+\infty}$$ 
converges weakly (in the weak $L^2$-sense) to a section $F\in L^2 (X, K_X\otimes L)$.

\medskip

\textbf{Step 3: Final conclusion}

We first prove that $F$ is holomorphic and satisfies \eqref{normcontrol}.
In fact, thanks to \eqref{estimateweight1} and \eqref{keyestimateoptial}, we have 
\begin{equation}\label{addestimate2}
\int_{X}  |\beta_m|^2 _{\omega, h_{a_m}} e^{-\Psi} dV_{X,\omega}\leq C_3
\end{equation}
for some uniform constant $C_3$.
Since $\frac{m}{a_m}$ tends to $0$, \eqref{partialequation} and \eqref{addestimate2} imply that 
$D^{''} (\gamma_m -(1-b_m '\circ \Psi)\widetilde{f})$ tends to $0$. 
Therefore
$F\in H^0 (X, K_X \otimes L)$.

\smallskip

As $\{\varphi_k\}_{k=1}^{+\infty}$ is a decreasing sequence, for every $k_0\in \mathbb{N}$ fixed, we have
\begin{equation}\label{decreasingsequence}
\int_X  c_A (- b_m \circ\Psi)  |\gamma_m|^2 _{\omega, h_{k_0}}  dV_{\omega, X}
\leq \int_X  c_A (- b_m \circ\Psi)|\gamma_m|^2 _{\omega, h_k}  dV_{\omega, X}
\end{equation}
for every $k \geq k_0$. Combining this with \eqref{addestimate}, we get
$$\varliminf_{m\rightarrow +\infty}\int_X  c_A (- b_m \circ\Psi)  |\gamma_m |^2_{\omega, h_{k_0}} d V_{X,\omega} \leq 
e^{-\lim\limits_{t\rightarrow +\infty} u(t)} \cdot \int_Z \frac{|f|_{\omega, h_0}^2 e^{-\varphi_k}} {|\Lambda^r (d v)|^2} dV_{Z,\omega}$$
Applying Fatou's lemma to the above inequality, we obtain 
$$\int_X c_A (-\Psi )  |F|^2 _{\omega, h_{k_0}}  dV_{\omega, X}
\leq e^{-\lim\limits_{t\rightarrow +\infty} u(t)} \int_Z \frac{|f|_{\omega, h}^2} {|\Lambda^r (d v)|^2} dV_{Z,\omega},$$
and \eqref{normcontrol} is proved by letting $k_0\rightarrow +\infty$. 

\medskip

Let $\{U_i \}$ be a Stein cover of $X$.
To finish the proof of the theorem, it remains to prove that $F|_{U_i \cap Z}=f$ for every $i$. 
Since $\beta_m$ is $\overline{\partial}$-closed, on the Stein open set $U_i$, we can find a function 
$w_m$ such that $\overline{\partial} w_m = \beta_m$ and
$$\int_{U_i} | w_m |^2 _{\omega, h_{a_m}} e^{-\Psi}  dV_{X,\omega} \leq C_4 \int_{U_i} | \beta_m |^2 _{\omega, h_{a_m}} e^{-\Psi} dV_{X,\omega} \leq C_4 \cdot C_3 .$$
for some uniform constant $C_4$.
Then 
$$F_m := (1- b_m ' \circ \Psi) \cdot\widetilde{f} - \gamma_{m} -(\frac{m}{a_m})^{\frac{1}{2}} \cdot w_m$$
is a holomorphic function on $U_i$ and $F_m \rightharpoonup F$ on $U_i$.
As $F_m |_{U_i \cap Z} =f$ by construction, we know that $F |_{U_i \cap Z}=f$.
The theorem is proved.
\end{proof}

We complete here the proof of Theorem \ref{mainthmoptimal} by establishing the claim in Step 2.

\begin{lemme}\label{L2estimateoptimal}
The claim in Theorem \ref{mainthmoptimal} is true.
\end{lemme}

\begin{proof}

{\em Step 1: Approximation}

Since $b_{m}$ is not smooth, we construct first a smooth approximation of $b_m$.
Let $m , k$ be two fixed constants. Set 
$$v_{\ls} (t):= \int_{-\infty}^t \int_{-\infty}^{t_1} \frac{1}{1-2\ls} \textbf{1}_{\{-m-1+\ls < s < -m-\ls \}}*\rho_{\frac{\ls}{4}}ds dt_1$$
$$-\int_{-\infty}^0 \int_{-\infty}^{t_1} \frac{1}{1-2\ls}\textbf{1}_{\{-m-1+\ls < s < -m-\ls \}}*\rho_{\frac{\ls}{4}}ds dt_1$$
where $\rho_{\frac{\ls}{4}}$ is the kernel of convolution satisfying $\supp (\rho_{\frac{\ls}{4}})\subset (-\frac{\ls}{4}, \frac{\ls}{4})$.
It is easy to check that $v_\ls (t)$ is a smooth approximation of $b_m (t)$.
Set 
$$\eta_{\ls} := s(- v_\ls \circ \Psi), \qquad \phi_{\ls} := u (-v_{\ls} \circ \Psi), \qquad \widetilde{h}_\ls := h_k\cdot e^{-\Psi -\phi_\ls} $$
and 
$$B_\ls := [\eta_\ls i \Theta_{\widetilde{h}_{m,k}} -i \partial\overline{\partial} \eta_\ls - 
i (\lambda_\ls)^{-1} \partial \eta_\ls \wedge \overline{\partial}\eta_\ls, \Lambda_{\omega}] ,$$
where $\Lambda_{\omega}$ is the contraction with respect to $\omega$.
Then $\eta_{\ls} , \phi_\ls, B_\ls$ tend to $\eta_m, \phi_m, B_{m, k}$.

\medskip

{\em Step 2: $L^2$ estimate}

By using the estimates in \cite[page 1180]{GZ15a}, we know that 
\begin{equation}\label{curvature}
B_{m, k} :=[\eta_m (i\Theta_{\widetilde{h}_{m,k}} (L)) - i d' d'' \eta_m -\lambda_m ^{-1} id'\eta_m \wedge d''\eta_m, \Lambda_\omega]
\end{equation}
satisfies
$$B_{m, k} \geq ( b_m ''\circ \Psi )\cdot [\partial\Psi \wedge \overline{\partial}\Psi, \Lambda_\omega] - \frac{\eta_m}{k} \Id .$$
Combining this with \eqref{estimateweight2}, we have
$$B_{m, k} \geq ( b_m ''\circ \Psi )\cdot [\partial\Psi \wedge \overline{\partial}\Psi, \Lambda_\omega] - \frac{C\cdot m}{k} \Id .$$
for some uniform constant $C$.
Therefore, for every form $\alpha\in C_c ^{\infty} (X, \wedge^{n,1}T_X ^*\otimes L)$, 
we have \footnote{We refer to \cite[5.1]{GZ15a} for a detailed calculus.}
\begin{equation}\label{importantetimate}
 \| (\eta_\ls + \lambda_\ls)^{\frac{1}{2}} (D'')^* \alpha\|_{\widetilde{h}_k} ^2 + \| (\eta_\ls)^{\frac{1}{2}} D'' \alpha\|_{\widetilde{h}_k} ^2 \geq  \langle B_\ls \alpha, \alpha \rangle_{\widetilde{h}_k} .
\end{equation}
and
\begin{equation}\label{lessimportant}
\langle (B_\ls + \frac{C\cdot m}{k} \Id) \alpha, \alpha \rangle_{\widetilde{h}_\ls } \geq 
(v'' _{\ls} \circ \Psi ) \langle [\partial \Psi \wedge \overline{\partial} \Psi, \Lambda_{\omega}]\alpha, \alpha \rangle_{\widetilde{h}_\ls}
\end{equation}

\medskip

By applying a standard $L^2$-estimate (cf. appendix), we can find $\gamma_{\ls}$ and $\beta_{\ls}$
such that 
\begin{equation}\label{l2estimatestep2}
D^{''}\gamma_\ls + (\frac{m}{k})^{\frac{1}{2}} \beta_\ls = g_m 
\end{equation}
and
\begin{equation}\label{l2estimateStep2}
\int_X \frac{|\gamma_\ls|_{\omega, \widetilde{h}_\ls} ^2 }{\eta_\ls + \lambda_\ls} d V_{X,\omega} 
+\frac{1}{2C}\int_X |\beta_\ls|_{\omega, \widetilde{h}_\ls} ^2 d V_{X,\omega}
\end{equation}
$$\leq \int_X \langle (B_{\ls}+\frac{2C\cdot m}{k})^{-1} g_m , g_m  \rangle_{\omega, \widetilde{h}_{m,k}} dV_{X,\omega} . $$
By letting $\ls\rightarrow 0$, we can find $\gamma_{m , k}$ and $\beta_{m, k}$, 
such that 
\begin{equation}\label{l2estimatenewadd}
 D^{''}\gamma_{m,k} + (\frac{m}{k})^{\frac{1}{2}} \beta_{m,k} = g_m 
\end{equation}
and
\begin{equation}\label{l2estimatestep21}
\int_X \frac{|\gamma_{m,k}|_{\omega, \widetilde{h}_{m,k}} ^2}{ \eta_m + \lambda_m} dV_{X,\omega} 
+\frac{1}{2C}\int_X |\beta_{m,k}|_{\omega, \widetilde{h}_{m,k}} ^2 dV_{X,\omega}
\end{equation}
$$\leq \int_X \langle (B_{m , k}+\frac{2C\cdot m}{k})^{-1} g_m , g_m \rangle_{\omega, \widetilde{h}_{m,k}} dV_{X,\omega} .$$

\medskip

{\em Step 3: Final conclusion}

We first estimate the right hand side of \eqref{l2estimatestep21}.
By the construction of $g_m $ and \eqref{lessimportant}, we have 
\begin{equation}\label{Demestimate}
\int_X \langle (B_{m,k}+\frac{2C\cdot m}{k})^{-1} g_m , g_m \rangle_{\omega, \widetilde{h}_{m , k}} dV_{X,\omega}
\end{equation}
$$\leq 
 \int_X ( b_m ''\circ\Psi )\cdot  |\widetilde{f}|^2 _{\omega, \widetilde{h}_{m ,k}}  dV_\omega 
+ \frac{C\cdot k}{m}\int_X (1- b'_m \circ \Psi)|D''\widetilde{f}|^2 _{\omega, \widetilde{h}_{m, k}} dV_{X, \omega} .
$$
Since $(1- b'_m\circ\Psi) (z)= 0$ on $ \{ \Psi \geq -m \}$, we have 
$$\int_X (1- b'_m \circ \Psi)|D''\widetilde{f}|^2 _{\omega, \widetilde{h}_{m, k}} dV_{X, \omega} 
\leq \int_{ \{ \Psi \leq -m \}} |D''\widetilde{f} |_{\omega, \widetilde{h}_{m,k}} ^2 dV_{X,\omega} .$$
We use the following key estimate \cite[Lemma 3.1]{Yi12}: by Hölder inequality, we have
\begin{equation}\label{holder}
\int_{\{ \Psi \leq -m\}} \frac{|D''\widetilde{f} |_{\omega,h} ^2}{ |v|_{h_E} ^{2r}} dV_{X, \omega} 
\end{equation}
$$\leq (\int_{\{ \Psi \leq -m\}} \frac{|D''\widetilde{f} |_{\omega, h_0} ^2 e^{- (1+\sigma) \varphi}}{ |v|_{h_E} ^{2r} (\ln |v|_{h_E})^2 } dV_{X,\omega} )^{\frac{1}{1+\sigma}} \cdot 
(\int_{\{ \Psi \leq -m\}}\frac{|D''\widetilde{f} |_{\omega, h_0} ^2  (\ln |v|_{h_E} )^\frac{2}{\sigma}}{ |v|_{h_E} ^{2r}} dV_{X,\omega})^\frac{\sigma}{1+\sigma} .$$
As $D''\widetilde{f} =0$ on $Z$ by construction, we have 
$$\lim_{m\rightarrow +\infty} \int_{\{ \Psi \leq -m\}}\frac{|D''\widetilde{f} |_{\omega, h_0} ^2  (\ln |v|_{h_E} )^\frac{2}{\sigma}}{ |v|_{h_E} ^{2r}} dV_{X,\omega} =0.$$ 
Combining this with \eqref{holder} and \eqref{l2controlnew}, we obtain 
$$\lim_{m\rightarrow +\infty}\int_{\{ \Psi \leq -m\}} \frac{|D''\widetilde{f} |_{\omega,h} ^2}{ |v|_{h_E} ^{2r}} dV_{X,\omega} =0 .$$
As a consequence, we can find a sequence $a_m \rightarrow +\infty$ such that 
\begin{equation}\label{limitconp}
\lim_{m\rightarrow +\infty} \frac{m}{a_m} = 0 \qquad\text{and}\qquad
\lim\limits_{m\rightarrow +\infty }\frac{a_m}{m}\int_{ \{\Psi \leq -m\}} |D''\widetilde{f} |_{\omega,\widetilde{h}_{m,a_m}} ^2 dV_{X,\omega} =0 .
\end{equation}
Applying \eqref{limitconp} to \eqref{Demestimate}, we obtain
\begin{equation}\label{lemmestim}
\varlimsup_{m\rightarrow +\infty} \int_X \langle (B_{m,a_m}+\frac{2 C\cdot m}{ a_m})^{-1} g_m , g_m \rangle_{\omega, \widetilde{h}_{m , a_m}}
dV_{X,\omega}
\end{equation}
$$\leq \varlimsup_{m\rightarrow +\infty} \int_X ( b_m ''\circ\Psi )\cdot  |\widetilde{f}|^2 _{\omega, \widetilde{h}_{m ,k}}  dV_{X,\omega} \leq 
e^{-\lim\limits_{t\rightarrow +\infty} u(t)} \cdot \int_Z \frac{|f|_{\omega, h}^2} {|\Lambda^r (d v)|^2} dV_{Z,\omega}.$$

\medskip

Finally, we take $\gamma_{m} = \gamma_{m, a_m}$ and $\beta_m =\beta_{m, a_m}$ in \eqref{l2estimatenewadd}. 
Then \eqref{l2estimatestep21} and \eqref{lemmestim} imply
$$\varlimsup_{m\rightarrow +\infty} 
(\int_X |\gamma_m |_{\omega,\widetilde{h}_{m, a_m}} ^2 (\eta_m + \lambda_m)^{-1} dV_{X,\omega} 
+\frac{1}{2C}\int_X |\beta_m |_{\omega, \widetilde{h}_{m,a_m}} ^2 dV_{X,\omega})$$
$$\leq 
e^{-\lim\limits_{t\rightarrow +\infty} u(t)} \cdot \int_Z \frac{|f|_{\omega, h}^2} {|\Lambda^r (d v)|^2} dV_{Z,\omega}.$$
The lemma is proved.
\end{proof}

\section{Applications}

\subsection{Some direct applications}

As pointed out in Remark \ref{specialcaseremark} in the introduction, Theorem \ref{mainthmoptimal} implies that

\begin{cor}\label{optimalestimatemaintext}
Let $(X,\omega)$ be a K\"ahler manifold with a proper map $\pr: X\rightarrow \Delta$ to a ball
$\Delta\subset \mathbb{C}^1$ centered at $0$ of radius $R$. 
Let $(L, h)$ be a holomorphic line bundle over $X$ equipped with 
a hermitian metric (maybe singular) $h$ such that $i\Theta_h (L)\geq 0$.
Suppose that $X_0 := \pr^{-1}(0)$ is a smooth subvariety of codimension $1$. 
Let $f\in H^0 (X_0 , K_{X_0} + L)$. 
Then there exists a section $F\in H^0 (X, K_{X} + L)$
such that
\begin{equation}\label{optimalestimate}
\frac{1}{\pi R^2}\int_{X} |F|_{\omega,h} ^2 d V_{X,\omega}\leq  \int_{X_0} |f|_{\omega,h} ^2 dV_{X_0, \omega}
\end{equation}
and $F|_{X_0}= \pr^* (dt)\wedge f$, where $t$ is the standard coordinate of $\mathbb{C}^1$.
\end{cor}

By the same arguments as in \cite[A.1]{BP10}, Corollary \ref{optimalestimatemaintext} implies the following result:

\begin{cor}\label{extensionoptimalm}
Let $(X,\omega)$ be a K\"ahler manifold and $\pr: X\rightarrow \Delta$ be a proper map to the ball 
$\Delta\subset \mathbb{C}^1$ centered at $0$ of radius $R$. 
Let $L$ be a holomorphic line bundle over $X$ equipped with 
a hermitian metric (maybe singular) $h= h_0 \cdot e^{-\varphi}$ such that $i\Theta_{h} (L)\geq 0$
in the sense of current, where $h_0$ is a smooth metric and $\varphi$ is a quasi-psh function on $X$.
Suppose that $X_0 := \pr^{-1}(0)$ is smooth of codimension $1$. 
Let $f\in H^0 (X_0 , m K_{X_0} \otimes L)$. We suppose that 
$$\int_{X_0}|f|_{\omega, h} ^{\frac{2}{m}} d V_{X_0,\omega}< + \infty$$
and there exists a $F\in H^0 (X, m K_{X} \otimes L)$
such that
$$F |_{X_0}= f \otimes \pr^* (dt ^{\otimes m}) \qquad\text{and}\qquad \int_{X} |F|_{\omega, h} ^{\frac{2}{m}} d V_{X,\omega}< + \infty ,$$
where $t$ is the standard coordinate of $\mathbb{C}^1$.
Then there exists a $\widetilde{F}\in H^0 (X, m K_X \otimes L)$
such that
\begin{equation}\label{optimalestimate2}
\widetilde{F} |_{X_0}=  f \otimes \pr^* (dt ^{\otimes m})  \qquad\text{and} \qquad
\frac{1}{\pi R^2}\int_{X}  |\widetilde{F}|_{\omega, h} ^{\frac{2}{m}} d V_{X,\omega} 
\leq  \int_{X_0} |f|_{\omega, h} ^{\frac{2}{m}} d V_{X_0, \omega} .
\end{equation}
\end{cor}

\begin{proof}
The proof given here follows closely \cite[A.1]{BP10}.
Set 
$$C_1 := \int_{X_0} |f|_{\omega, h} ^{\frac{2}{m}} dV_{X_0, \omega}
\qquad\text{and}\qquad C_2 :=\frac{1}{\pi R^2}\int_{X} |F|_{\omega, h} ^{\frac{2}{m}} d V_{X, \omega}.$$
If $C_2 \leq C_1$, then $F$ satisfies \eqref{optimalestimate2} and the corollary is proved.
If $C_1 < C_2$,
since $F$ is holomorphic, we can apply Corollary \ref{optimalestimatemaintext} with weight 
$$\varphi_1 := \frac{m-1}{m}\ln |F|_{\omega, h_0} ^2 + \frac{1}{m}\varphi$$
on the line bundle $(m-1)K_X \otimes L$, and obtain a new extension $F_1$ of $f$ satisfying
\begin{equation}\label{newextension}
\frac{1}{\pi R^2}\int_X \frac{|F_1|_{\omega, h} ^2}{|F|_{\omega, h} ^\frac{2(m-1)}{m}} d V_{X, \omega}
\leq \int_{X_0} \frac{|f|_{\omega,h} ^2}{|f|_{\omega, h} ^\frac{2(m-1)}{m}} d V_{X_0, \omega}
=\int_{X_0} |f|_{\omega, h} ^{\frac{2}{m}} dV_{X_0 ,\omega} .
\end{equation}
By H\"{o}lder's inequality, we have
$$\frac{1}{\pi R^2}\int_X |F_1|_{\omega, h} ^{\frac{2}{m}}  d V_{X, \omega}
\leq (\frac{1}{\pi R^2} \int_X |F|_{\omega, h} ^{\frac{2}{m}} d V_{X, \omega})^{1-\frac{1}{m}}\cdot 
(\frac{1}{\pi R^2}\int_X \frac{|F_1|_{\omega, h} ^2}{|F|_{\omega, h} ^\frac{2(m-1)}{m}} d V_{X, \omega})^{\frac{1}{m}} .$$
Combining with \eqref{newextension}, we have
\begin{equation}\label{recurrence}
\frac{1}{\pi R^2}\int_X |F_1|_{\omega, h} ^{\frac{2}{m}} d V_{X, \omega}\leq (C_2) ^{1- \frac{1}{m}} (C_1)^{\frac{1}{m}} .
\end{equation}
We can repeat the same argument with $F$ replaced by $F_1$, etc.
We obtain thus a sequence $\{F_i\}_{i=1} ^{+\infty}\subset H^0 (X, m K_X \otimes L)$,
and 
\begin{equation}\label{repeat}
\frac{1}{\pi R^2}\int_X |F_i|_{\omega, h} ^{\frac{2}{m}} d V_{X, \omega}\leq 
(\frac{1}{\pi R^2}\int_X |F_{i-1}|_{\omega, h} ^{\frac{2}{m}} d V_{X, \omega})\cdot (C_1)^{\frac{1}{m}}
\end{equation}

\medskip

If there exists an $i\in \mathbb{N}$ such that $F_i$ satisfies \eqref{optimalestimate2}, then Corollary \eqref{extensionoptimalm} is proved.
If not, thanks to \eqref{repeat}, we have
\begin{equation}\label{recurrence2}
\frac{1}{\pi R^2}\int_X |F_i|_{\omega, h} ^{\frac{2}{m}} d V_{X, \omega}\searrow C_1 .
\end{equation}
By passing to a subsequence, $F_i$ tends to a section $\widetilde{F} \in H^0 (X, m K_X \otimes L)$, 
and $\widetilde{F} \mid_{Z} =f$.
By Fatou lemma, \eqref{recurrence2} implies that 
$$\frac{1}{\pi R^2}\int_{X} |\widetilde{F}|_{\omega, h} ^{\frac{2}{m}}  d V_{X, \omega} \leq C_1 .$$
Corollary \ref{extensionoptimalm} is proved.
\end{proof}

\subsection{Positivity of $m$-relative Bergman Kernel metric}

We first recall the definition of $m$-relative Bergman Kernel metric (cf. \cite[A.2]{BP10}, \cite{BP08}, \cite{Kaw85}, \cite{Tsu07}).
Let $p: X\rightarrow Y$ be a surjective map between two smooth manifolds and 
let $(L, h_L )$ be a line bundle on $X$ equipped with a hermitian metric $h_L$.
Let $x\in X$ be a point on a smooth fiber of $p$.
We first define a hermitian metric $h$ on $-(m K_{X/Y} +L)_x$ by 
\begin{equation}\label{dualbergman}
\|\xi\|_h ^2 :=\sup \frac{|\xi (\tau(x) )|^2}{(\int_{X_{p(x)}} |\tau|_{\omega,h_L} ^{\frac{2}{m}} d V_{X_{p(x)}, \omega})^m} ,
\end{equation}
where the "sup" is taken over all sections $\tau\in H^0 (X_{p(x)}, m K_{X/Y}+ L)$.
The $m$-relative Bergman Kernel metric $h^{(m)}_{X/Y}$ on $m K_{X/Y} +L$ is defined to be the dual of $h$. 

Although the construction of the metric $h^{(m)}_{X/Y}$ is fiberwise and only defined on the smooth fibers, 
by using the positivity of direct image arguments, \cite[Thm 0.1]{BP10} proved that:

\begin{thm}\label{thmbergmankernelpositivity}\cite[Thm 0.1]{BP10}
Let $p: X\rightarrow Y$ be a fibration between two projective manifolds, and let $\omega$ be a K\"{a}hler metric on $X$.
Let $L\rightarrow X$ be a line bundle endowed with a metric (maybe singular) $h$ such that $i\Theta_{h} (L) \geq 0$.
Suppose that there exists a generic point $z\in Y$ and a section $u\in H^0 (X_z,  m K_{X/Y} + L)$ such that 
$$\int_{X_z} |u|_{\omega, h} ^{\frac{2}{m}} d V_{X_z , \omega} < +\infty .$$
Then the line bundle $m K_{X/Y} + L$ admits a metric with positive curvature current. 
Moreover, this metric equals to the fiberwise $m$-Bergman kernel metric (with respect to $h$ ) on the generic fibers of $p$.
\end{thm}

An alternative proof of Theorem \ref{thmbergmankernelpositivity} 
is given by using the optimal extension proved in \cite[Thm 2.1, Cor 3.7]{GZ15a}.
We should remark that, if $\varphi_L$ has arbitrary singularity, the proof of in \cite[Thm 0.1]{BP10} uses the existence of ample line on $X$.
Therefore the assumption that $p$ is a projective map is essential in the proof of Theorem \ref{thmbergmankernelpositivity} 
in \cite[Thm 0.1]{BP10}.
However, as pointed out by M. P\u{a}un, 
since the optimal extension proved in Corollary \ref{extensionoptimalm} is without projectivity assumption, 
we can use Corollary \ref{extensionoptimalm} to generalize Theorem \ref{thmbergmankernelpositivity} 
to arbitrary compact K\"ahler fibrations, 
by using the same arguments in \cite[Cor 3.7]{GZ15a}.
For the reader's convenience, we give the proof of this generalization in this subsection.

To begin with, we first prove the following lemma, which uses the recent important result \cite{GZ15b}.

\begin{lemme}\label{coherent}
Let $\varphi$ be a psh function on a Stein open set $U$.
Set: 
$$\mathcal{I}_m (\varphi)_x := \{f\in \mathcal{O}_x | \int_{U_x} |f|^{\frac{2}{m}} e^{-\frac{\varphi}{m}} < +\infty \} .$$
Then $\mathcal{I}_m (\varphi)$ is a coherent sheaf.
\end{lemme}

\begin{proof}
We first prove the lemma under the assumption that $\varphi$ has analytic singularities.
In this case,
Let $\pi: \widetilde{U} \rightarrow U$ be a resolution of singularities of $\varphi$, i.e., $\varphi\circ\pi$ can be written locally as
$$\varphi\circ\pi = \sum_i a_i \ln (|s_i|) +O(1),$$
where $s_i$ are holomorphic functions on $\widetilde{U}$ and $\bigcup\limits_i \Div (s_i)$ is normal crossing.
We suppose that $K_{\widetilde{U}}= K_X + \sum\limits_i b_i\cdot  E_i$ and $\sum\limits_i a_i\cdot \Div (s_i) = \sum\limits_i c_i\cdot E_i$.
Let $k_i$ be the minimal number in $\mathbb{Z}^+$ such that $k_i \cdot \frac{2}{m} > \frac{c_i}{m} - 2 b_i -2$.
It is easy to check that $\mathcal{I}_m (\varphi)=\pi_* (\mathcal{O} (-\sum\limits_i k_i\cdot E_i))$.
Therefore $\mathcal{I}_m (\varphi)$ is a coherent sheaf.

We now prove the lemma for arbitrary psh functions. 
Thanks to \cite[15.B]{Dem12}, we can find a sequence of quasi-psh $\varphi_k$ with analytic singularities and a sequence $\delta_k \rightarrow 0^+$,
such that

{\em (i):} $\varphi_k$ decrease to $\varphi$.

{\em (ii):} $\int_{\{\frac{\varphi}{m} < \frac{ (1+\delta_k) \varphi_k}{m} +a_k\}} e^{-\frac{\varphi}{m}}  < +\infty$ 
(cf. \cite[proof of Thm 15.3, Step 2]{Dem12}) for certain constant $a_k$.

\noindent As a consequence, we have $\sI_m ((1+\delta_k) \varphi_k) \subset \sI_m (\varphi)$.
Since we proved that 
$$\sI_m ((1+\delta_k) \varphi_k)$$ 
are coherent, by the Noetherien property of coherent sheaf,
$\bigcup\limits_{k=1}^{+\infty} \sI_m ((1+\delta_k) \varphi_k)$ is also coherent and 
$$\bigcup_{k=1}^{+\infty} \sI_m ((1+\delta_k) \varphi_k)\subset \sI_m (\varphi).$$
To prove the lemma, it is sufficient to prove that for every $f\in  \sI_m (\varphi)$, we can find a $k\in\mathbb{N}$, such that 
$f\in  \sI_m ((1+\delta_k) \varphi_k)$. 

Let $f$ be a holomorphic germ of $(\sI_m (\varphi) )_x$.
Then
$$\int_{U_x} |f|^2 e^{-\frac{\varphi}{m} -\frac{2(m-1)\ln |f|}{m}} <+\infty ,$$
for some neighborhood $U_x$ of $x$.
By \cite{GZ15b}, there exists some $\delta >0$, such that 
$$\int_{U_x} |f|^2 e^{-\frac{(1+\delta)\varphi}{m} -\frac{2(1+\delta)(m-1)\ln |f|}{m}} < +\infty .$$
Replacing $U_x$ by a smaller neighborhood $U_x '$ of $x$, we have 
\begin{equation}\label{addedequ}
\int_{U_x '} |f|^{\frac{2}{m}} e^{-\frac{(1+\delta)\varphi}{m}} < +\infty .
\end{equation}
We take a $k\in\mathbb{N}$, such that $\delta_k <\delta$.
Thanks to $(i)$ and \eqref{addedequ}, we have 
$$\int_{U_x '} |f|^{\frac{2}{m}} e^{-\frac{(1+\delta_k)\varphi_k}{m}} < +\infty .$$
Therefore $f\in  \sI_m ((1+\delta_k) \varphi_k)$ and the lemma is proved.
\end{proof}

We now generalize \cite[Thm 0.1]{BP10} to arbitrary proper K\"ahler fibrations. The proof is almost the same as \cite[Cor 3.7]{GZ15a}.

\begin{thm}
Let $p: X\rightarrow Y$ be a proper fibration between two K\"ahler manifolds and let $\omega$ be a K\"{a}hler metric on $X$.
Let $L\rightarrow X$ be a line bundle endowed with a metric (maybe singular) $h= h_0 \cdot e^{-\varphi}$ such that $i\Theta_{h} (L) \geq 0$
in the sense of current, where $h_0$ is a smooth metric and $\varphi$ is a quasi-psh function on $X$.

Suppose that there exists a generic point $z\in Y$ and a $u\in H^0 (X_z,  (K_{X/Y})^m \otimes L)$ such that 
$$\int_{X_z} |u|_{\omega, h} ^{\frac{2}{m}} d V_{X_z , \omega} < +\infty \qquad\text{and}\qquad u \not\equiv 0.$$
Then the line bundle $( K_{X/Y})^m \otimes L$ admits a metric with positive curvature current. 
Moreover, this metric equals to the fiberwise $m$-Bergman kernel metric (with respect to $h$ ) on the generic fibers of $p$.
\end{thm}

\begin{proof}
By Lemma \ref{coherent}, $p_* (( K_{X/Y})^m \otimes L \otimes \mathcal{I}_m (\varphi))$ is coherent.
Using \cite{Fle81} (cf. also \cite[Thm 10.7, page 47]{BDIP02}), there exists a subvariety $Z$ of $Y$ of codimension at least $1$
such that $p$ is smooth on $Y\setminus Z$ and for every point $t \in Y\setminus Z$, we have
$$\dim H^0 (X_t, ( K_{X/Y})^m \otimes L \otimes \mathcal{I}_m (\varphi)|_{X_t}) 
= \rank p_* (( K_{X/Y})^m \otimes L \otimes \mathcal{I}_m (\varphi)) ,$$
where $\mathcal{I}_m (\varphi)|_{X_t}$ is the restriction of the coherent sheaf $\mathcal{I}_m (\varphi)$ on $X_t$.
By local extension theorem, we know that $\mathcal{I}_m (\varphi |_{X_t}) \subset \mathcal{I}_m (\varphi) |_{X_t}$.
As a consequence, for every Stein neighborhood $U$ of $t \in Y\setminus Z$, the fibration $p: p^{-1} (U) \rightarrow U$
and the point $t$ satisfy the conditions in Corollary \ref{extensionoptimalm}.

Let $h^{(m)}$ be the fiberwise $m$-Bergman kernel metric on $p^{-1}(Y\setminus Z) \rightarrow Y\setminus Z$ 
(cf. construction in the beginning of this subsection).
For every $x\in p^{-1}(Y\setminus Z)$, 
we now estimate the curvature of $h^{(m)}$ near $x$.
Let $e$ be a local coordinate of $ (K_{X/Y})^m \otimes L$ near $x$. 
Let 
\begin{equation}\label{localbergman}
B (z) := \sup \frac{|u^0 (z)|^2}{ (\int_{X_{p(z)}} |u|_{\omega,h} ^{\frac{2}{m}}  d V_{X_{p(z)}, \omega})^{m}} ,
\end{equation}
where $u= u^0 \cdot e$ and the "sup" is taken over all sections $u\in H^0(X_{p(z)}, (K_{X/Y})^m \otimes L\otimes \mathcal{I}_m (\varphi) )$.
Thanks to \eqref{dualbergman}, 
to prove that the curvature of $h^{(m)}$ is positive near $x$, 
it is sufficient to prove that
$\ln B (z)$ is $\psh$ near $x$.

For every fixed point $z$ near $x$, we can find a section
$u_1\in H^0(X_{p(z)}, (K_{X/Y})^m \otimes L\otimes \mathcal{I}_m (\varphi) )$ 
such that 
$$B (z) = \frac{|u_1 ^0 (z)|^2}{ (\int_{X_{p(z)}} |u_1|_{\omega,h} ^{\frac{2}{m}}  d V_{X_{p(z)}, \omega})^{m}} .$$
Let $\Delta_r$ be a one dimensional radius $r$ disc in $Y$ centered at $p(z)$,
and $\Delta_r '$ be a one dimensional disc in $X$ passing through $z$ and $p (\Delta_r ') =\Delta_r$.
Thanks to Proposition \ref{extensionoptimalm}, 
there exists an extension of $u_1$: $ U_1\in H^0 (p^{-1}(\Delta_r),  (K_X)^m \otimes L\otimes \mathcal{I}_m (\varphi) )$, 
such that
\begin{equation}\label{optimalfinal}
 \frac{1}{\pi r^2}\int_{p^{-1}(\Delta_r) } |U_1 |_{\omega, h} ^{\frac{2}{m}} dV_{X,\omega}
 \leq  \int_{X_{p(z)}} |u_1|_{\omega, h} ^{\frac{2}{m}} d V_{X_{p(z)}, \omega} .
\end{equation}
Set $\widetilde{u}_1 :=  U_1 / (d t)^{m} \in H^0 (p^{-1}(\Delta_r),  (K_{X/Y})^m \otimes L\otimes \mathcal{I}_m (\varphi) )$
and $\widetilde{u}_1 ^0 := \frac{\widetilde{u}_1 }{e}$, where $t$ is coordinate of $\Delta_r$. 
By the definition of $B (z)$, we have
$$\frac{1}{\pi r^2} \int_{\Delta_r '} \ln B(x) p^* (d' t \wedge d'' t)
\geq \frac{1}{\pi r^2} \int_{\Delta_r '}\ln \frac{|\widetilde{u}_1 ^0 (x)|^2}{(\int_{X_{p(x)}} |\widetilde{u}_1 (x)|_{\omega, h} ^{\frac{2}{m}}  d V_{X_{p(z)}, \omega} )^{m}} p^* (d' t \wedge d'' t)$$
$$\geq \frac{1}{\pi r^2}\int_{\Delta_r '}\ln |\widetilde{u}_1 ^0|^2 p^* (d' t \wedge d'' t) 
-\frac{m}{\pi r^2}\ln\int_{p^{-1}(\Delta_r) } |U_1|_{\omega, h} ^{\frac{2}{m}} d V_{X,\omega} .$$
Combining this with \eqref{optimalfinal} and the holomorphicity of $\widetilde{u}_1 ^0$, 
we obtain
$$\frac{1}{\pi r^2} \int_{\Delta_r '} \ln B(x) p^* (d' t \wedge d'' t)  \geq \ln B(z) .$$
Therefore, $\ln B (x)$ is $\psh$ in the horizontal direction.
By the convexity of $\ln |u^0 (x)|$ and the construction of $\ln B (x)$, 
$\ln B (x)$ is also $\psh$ in the fiberwise direction.
Therefore $\ln B (x)$ is $\psh$ on $p^{-1} (Y\setminus Z)$ and the curvature of $h^{(m)}$ is semi-positive on 
$p^{-1} (Y\setminus Z)$ (in the sense of currents).

Using the arguments in \cite[A.2]{BP10},
we now prove that $h^{(m)}$ can be extended to the whole $X$.
We first express $h^{(m)}$ locally as the potential form $e^{-\varphi_{X/Y}}$, where $\varphi_{X/Y}$ is a quasi-psh function outside the subvariety
$p^{-1}(Z)$. By the standard results in pluripotential theory, to prove that $h^{(m)}$ can be extended to $X$,
it is sufficient to prove the existence of a uniform constant $C$ such that 
\begin{equation}\label{extenionbound}
\varphi_{X/Y} \leq C \qquad \text{ on } X\setminus p^{-1}(Z).
\end{equation}
Let $U$ be a small open set in $X$. 
Let $B$ be the function on $U\setminus p^{-1}(Z)$ defined by \eqref{localbergman}.
Thanks to \eqref{dualbergman}, to prove \eqref{extenionbound}, it is equivalent to prove that
$B$ is uniformly bounded on $U\setminus p^{-1}(Z)$.
For every $z\in U\setminus p^{-1}(Z)$, we can find a $u_2 \in H^0(X_{p(z)}, (K_{X/Y})^m \otimes L\otimes \mathcal{I}_m (\varphi) )$ such that 
$$B (z) = |u_2 ^0 (z)|^2 \qquad\text{and}\qquad  
\int_{X_{p(z)}} |u_2 |_{\omega, h} ^{\frac{2}{m}} d V_{X_{p(z)}, \omega} =1 ,$$
where $u_2 ^0 :=\frac{u_2}{e}$.
Using Proposition \ref{extensionoptimalm}, we can find an extension $\widetilde{u}_2$ of $u_2$, such that
$$\int_{p^{-1}(p (U) ) }  |\widetilde{u}_2 |_{\omega, h} ^{\frac{2}{m}}  d V_{X, \omega}\leq C_U,$$
where the constant $C_U$ depends only on $U$.
By mean value inequality, we know that $ |u_2 ^0 (z) |$ is controlled by a constant depending only on $C_U$.
The theorem is thus proved.

\end{proof}

\section{Appendix}

For the reader's convenience, we give the proof of \eqref{l2estimatestep2} and \eqref{l2estimateStep2},
which is a rather standard estimate (cf. \cite[Prop. 12.4, Remark 12.5]{Dem12}, \cite{DP03} or \cite{Yi12}). 

Set $g:= g_m$, $\eta:= \eta_\ls$, $B:= B_{\ls, k}$ and $\delta :=\delta_k$ for simplicity.
Let $Y_k$ be a subvariety of $X$ such that $\varphi_k$ is smooth outside $Y_k$.
Then there exists a complete K\"ahler metric $\omega_1$ on $X\setminus Y_k$.
Set $\omega_s := \omega+ s\omega_1$.
Then $\omega_s$ is also a complete K\"ahler metric on $X\setminus Y_k$ for every $s> 0$. 

We apply the twist $L^2$-estimate (cf. \cite[12.A, 12.B]{Dem12}) 
for the line bundle $(L, \widetilde{h}_k)$ on $(X\setminus Y_k, \omega_s )$.
Thanks to \eqref{importantetimate} and \cite[Lemma 4.1]{GZ15a}, 
for every smooth $(n,1)$-form $v$ with compact support, 
we have 
\begin{equation}\label{appendix}
| \langle g ,v \rangle_{\omega_s} | ^2
\end{equation}
$$\leq (\int_{X\setminus Y_k} \langle (B +\frac{2C\cdot m}{k})^{-1}g, g\rangle dV_{\omega_s}) \cdot
(\| (\eta+\lambda)^{\frac{1}{2}} D''^* v\|_{\omega_s} ^2 +\frac{2C\cdot m}{k} \int_{X\setminus Y_k} \langle  v, v\rangle dV_{\omega_s})$$
Set $H_1 :=\|\cdot\|_{L^2} $, where the $L^2$-norm $\|\cdot\|_{L^2}$ is defined with respect to the metrics 
$\omega_s$ and $(L, \widetilde{h}_k)$.
Let $H_2 $ be a Hilbert space where the norm is defined by
$$\|f\|_{H_2} ^2 :=  \frac{2C\cdot m}{k} \int_{X\setminus Y_k}   |f|_{\widetilde{h}_k} ^2  dV_{\omega_s} .$$
By \eqref{appendix} and the Hahn-Banach theorem, we can construct a continuous linear map (cf. for example \cite[5.A]{Dem12})
$$H_1 \oplus H_2 \rightarrow \mathbb{C},$$ 
which is an extension of the application
$$((\eta+\lambda )^{\frac{1}{2}} D''^* v , v)\rightarrow \langle g ,v\rangle_{\omega_s} .$$
Therefore, there exist $f$ and $h$ such that
$$\langle g, v\rangle_{\omega_s} = \langle f, (\eta+\lambda )^{\frac{1}{2}} D''^* v\rangle_{\omega_s} + \frac{2C\cdot m}{k} \langle  h, v\rangle_{\omega_s}$$
and
$$\|f\|_{\omega_s} ^2 +\frac{2C\cdot m}{k}\| h\|_{\omega_s} ^2\leq \int_X (\langle B + \frac{2C\cdot m}{k} )^{-1} g, g\rangle dV_{\omega_s}$$
Let $\beta :=  2C(\frac{m}{k})^{\frac{1}{2}}\cdot h$ and $\gamma : = (\eta+\lambda)^{\frac{1}{2}} f$. 
Then
$$g= D'' \gamma + (\frac{m}{k})^{\frac{1}{2}} \beta$$
and
$$ \|\frac{\gamma}{(\lambda +\eta )^{\frac{1}{2}}}\|^2 _{( X\setminus Y_k, \omega_s )} +\frac{1}{2C}\|\beta\|^2 _{ ( X\setminus Y_k, \omega_s )} \leq \int_{X\setminus Y_k} \langle (B + \frac{2C\cdot m}{k})^{-1}g, g\rangle dV_{\omega_s}$$
Then \eqref{l2estimatestep2} and \eqref{l2estimateStep2} are proved by letting $s\rightarrow 0^+$.

\end{document}